\documentclass[12pt,a4paper]{article}
\usepackage{amsfonts}
\usepackage{setspace}
\usepackage{mathrsfs}

\usepackage{stmaryrd}
\usepackage{amsmath}
\usepackage{amssymb}

\usepackage{bbm}
\usepackage{mathrsfs}
\usepackage{graphicx}

\usepackage{enumerate}
\usepackage{theorem}

\makeatletter
\def\tank#1{\protected@xdef\@thanks{\@thanks
        \protect\footnotetext[0]{#1}}}
\def\bigfoot{

    \@footnotetext}
\makeatother

\newcommand{\ea}{\end{array}}
\newtheorem{theorem}{Theorem}[section]

{\theorembodyfont{\rmfamily}
}

\newenvironment{proof}{Proof.}

\begin{document}

\title{Note on information bias and efficiency of composite likelihood}
\author{{Ximing Xu},~~{Nancy Reid} and
{Libai Xu}}
\date{}
\maketitle

\begin{spacing}{2.0}
\begin{abstract}
Does the asymptotic variance of the maximum composite likelihood estimator of a parameter of interest always decrease when the nuisance parameters are known? Will a composite likelihood necessarily become more efficient by incorporating additional independent component likelihoods, or by using component likelihoods with higher dimension? In this note we show through illustrative examples that the answer to both questions is no, and indeed the opposite direction might be observed. The role of information bias is highlighted to understand the occurrence of these paradoxical phenomenon.
\medskip

\end{abstract}

\section{Introduction}
The likelihood function for a complex multivariate model may not be available or very difficult to evaluate, and a composite likelihood function constructed from low-dimensional marginal or conditional distributions has become a popular alternative (Varin, 2008; Varin, Reid \& Firth, 2011). Suppose $\mathbf{Y}$ is a $p$-dimensional random vector with probability density function $f(\mathbf{y};\theta)$, with a $q$-dimensional parameter vector $\theta\in\Theta$. Given a set of likelihood functions $L_k(\theta;\mathbf{y})$, $k=1,\ldots,K$, defined by the joint or conditional densities of some sub-vectors of $\mathbf{Y}$, the composite likelihood function (Lindsay, 1988) is defined as
\begin{equation*}
CL(\theta;\mathbf{y})=\prod_{k=1}^{K}L_k(\theta;\mathbf{y})^{w_k}, %\label{eq:cl}
\end{equation*}
where the $w_k$'s are nonnegative weights and the component likelihood $L_k(\theta;\mathbf{y})$ might depend only on a sub-vector of $\theta$. The choice of $L_k(\theta;\mathbf{y})$ and the weights $\{w_k\}$ is critical for improving the efficiency of the resulting statistical inference (Lindsay, 1988; Joe \& Lee, 2009; Lindsay, Yi \& Sun, 2011). In this paper we focus on the two most commonly used composite likelihood functions in literature, independence likelihood and pairwise likelihood, which are defined as $CL_{ind}(\theta;\mathbf{y})=\prod_{r=1}^{p}f(y_r;\theta\, )$ and $CL_{pair}(\theta;\mathbf{y})=\prod_{r=1}^{p-1}\prod_{s=r+1}^{p}f(y_r,y_s;\theta \,)$, respectively. Given a random sample $\{\mathbf{y}^{(1)},\ldots,\mathbf{y}^{(n)}\}$, where each $\mathbf{y}^{(i)}$ is a $p$-dimensional vector, the composite log-likelihood function is
\begin{equation*}
{c\ell}(\theta;\mathbf{y})=\sum_{i=1}^{n}{c\ell}(\theta;\mathbf{y}^{(i)})=\sum_{i=1}^{n} \log CL(\theta;\mathbf{y}^{(i)}),
\end{equation*}
and the maximum composite likelihood estimator (MCLE) is $\hat{\theta}_{CL}={\arg\max}_{\theta}\,{c\ell}(\theta;\mathbf{y})$.

In addition to the computational simplicity, the composite likelihood function has many appealing theoretical properties. In particular, under some regularity conditions, $\hat{\theta}_{CL}$ is consistent and asymptotically normally distributed with variance equal to the inverse of the Godambe information matrix: $G(\theta)=H(\theta)J^{-1}(\theta)H(\theta)$ (Lindsay, 1988; Varin, 2008; Xu \& Reid, 2011). Here $H(\theta)=E\{-\nabla_{\theta}u_c(\theta;\mathbf{y})\}$ is the sensitivity matrix, and $J(\theta)={\rm Var}_\theta\{u_c(\theta;\mathbf{y})\}$ is the variability matrix, with the composite score function $u_c(\theta; y)=\nabla_{\theta}{c\ell}(\theta;\mathbf{y})$. Throughout this paper we use $I(\theta)$ to denote the Fisher information matrix of the full likelihood function. Given two composite likelihood functions $CL_1(\theta;\mathbf{y})$ and $CL_2(\theta;\mathbf{y})$,  $CL_2(\theta;\mathbf{y})$ is said to be  \textit{more efficient} than $CL_1(\theta;\mathbf{y})$ if $CL_2(\theta;\mathbf{y})$ has a greater Godambe information matrix than $CL_1(\theta;\mathbf{y})$ in the sense of matrix inequality. It is well known that the full likelihood function is more efficient than any other composite likelihood function under regularity conditions (Godambe, 1960; Lindsay, 1988), i.e. $I(\theta)-G(\theta)$ is non-negative definite.
% Throughout this paper both $H(\theta)$ and $J(\theta)$ are assumed to be positive definite matrices.

In general, the second Bartlett identity does not hold for composite likelihood functions, i.e. $H(\theta)\neq J(\theta)$. After Lindsay (1982), we call a composite likelihood $CL(\theta;\mathbf{y})$ \textit{information-unbiased} if $H(\theta)=J(\theta)$, and \textit{information-biased}, otherwise. Composite likelihood-based  inferential tools have been developed for hypothesis testing (Chandler \& Bate, 2007; Pace, Salvan, \& Sartori, 2011) and model selection (Varin \& Vidoni, 2005; Gao \& Song, 2010). Information bias of a composite likelihood can make the resulting inference more difficult. For example, if the composite likelihood is information-unbiased,  the likelihood ratio statistic has the same asymptotic chi-square distribution as its full likelihood counterpart. On the other hand, if it is information-biased  the likelihood ratio statistic converges in distribution to a weighted sum of some independent $\chi^2$ random variables (Kent, 1982). Adjustments have been proposed to the information-biased composite likelihood ratio statistic such that the adjusted statistic has an asymptotic chi-square distribution (e.g., Chandler \& Bate, 2007; Pace, Salvan, \& Sartori, 2011).

The full likelihood function is information-unbiased, but an information-unbiased composite likelihood is not necessarily fully efficient. %For example, each component likelihood function $L_k(\theta;\mathbf{y})$ itself is information unbiased but    only contains a small piece of information of the full model.
In fact, any component likelihood function $L_k(\theta;\mathbf{y})$ is information-unbiased. More generally, any composite likelihood function as the product of component likelihoods with mutually uncorrelated score functions is information-unbiased. As an example, consider a $p$-dimensional vector $Y=(Y_{1},\ldots,Y_{p})^{T}$ with density function $f(y_{1},\ldots, y_{p};\theta)$, and defining $f(y_{1}\mid y_{0};\theta)=f(y_{1};\theta)$. It is easy to show that the covariance between the score function of $f(y_{i}\mid y_1,\ldots, y_{i-1};\theta)$ and the score function of $f(y_{j}\mid y_{1},\ldots, y_{j-1};\theta)$ is zero for any $i\neq j\in \{1,\ldots,p\}$. Hence any composite likelihood of the form $$\prod_{i\in \mathcal{A}} f(y_{i}\mid y_{1},\ldots, y_{i-1};\theta),$$ where $\mathcal{A}\subseteq \{1,\ldots,p\}$, is information-unbiased. Conversely, an information-biased composite likelihood function can be fully efficient. The pairwise likelihood function for the equal-correlated multivariate normal model in Section 2 is fully efficient when estimating the common variance $\sigma^2$ and the correlation coefficient $\rho$ (Cox and Reid, 2004), but it is not information-unbiased (Pace, Salvan, \& Sartori, 2011). A sufficient and necessary condition for a composite likelihood to be fully efficient is given in the following theorem.

\begin{theorem}{Theorem 1.}{}%
Suppose the full likelihood function $L(\theta\mid \mathbf{y})$ has the score function $u(\theta)$ and Fisher information $I(\theta)$. Then, for any composite likelihood function $CL(\theta\mid \mathbf{y})$ with the score function $u_c(\theta)$, sensitivity matrix $H(\theta)$, variability matrix $J(\theta)$ and Godambe information $G(\theta)$, $G(\theta)=I(\theta)$ if and only if $u(\theta)=H(\theta)J(\theta)^{-1}u_c(\theta)+b(\theta)$ with probability $1$ for a constant vector $b(\theta)$ with respect to the random vector $\mathbf{y}$.
\end{theorem}
\begin{proof}{Proof}{}%
It is easy to show that $H(\theta)=$Cov$(u_c(\theta), u(\theta))$ (Lindsay, 1988). As $I(\theta)=$Var$(u(\theta))$ and $J(\theta)=$Var$(u_c(\theta))$, the result follows as the difference $I(\theta)-G(\theta)=I(\theta)-H(\theta)J(\theta)^{-1}H(\theta)$ is the covariance matrix of $u(\theta)-H(\theta)J(\theta)^{-1}u_c(\theta)$.
\end{proof}

Theorem 1 with $b(\theta)=0$ gives a sufficient condition for the maximum composite likelihood estimator to coincide with the MLE (Kenne Pagui, Salvan \& Sartori, 2014), which is satisfied by the pairwise likelihood for closed exponential family models (Mardia et al., 2009). In particular, the equicorrelated multivariate normal model with unknown variance $\sigma^2$ and correlation coefficient $\rho$ belongs to the closed exponential family, and it has been shown that  $b(\theta)=0$ and $u(\theta)=H(\theta)J(\theta)^{-1}u_c(\theta)$ or equivalently, $u_c(\theta)=J(\theta)H(\theta)^{-1}u(\theta)$  for the pairwise likelihood function (Pace, Salvan, \& Sartori, 2011). Its application in more general exponential family models has been studied in Kenne Pagui, Salvan, \& Sartori (2014b).

In this paper we explore the impact of information bias on the composite likelihood inference in more detail. In Section 2 we show through the equicorrelated multivariate normal model that an information-biased composite likelihood may lead to less efficient estimates of the parameters of interest when the nuisance parameters are known. A sufficient condition is also provided for the occurrence of such a paradoxical phenomenon. We would expect that a more efficient composite likelihood can be obtained by incorporating additional independent component likelihoods or using higher dimensional component likelihoods. However such strategies do not always work for information-biased composite likelihood functions, as shown in Section 3. We conclude with a discussion in Section 4.

\section{Composite likelihood with known nuisance parameters}
In the presence of nuisance parameters, it is well known that the maximum likelihood estimator of the parameter of interest will have a smaller asymptotic variance when the nuisance parameters are known. It is easy to check that this also holds for information-unbiased composite likelihood functions. Suppose the $q$-dimensional parameter vector $\theta$ is partitioned as $\theta=(\psi, \lambda)$, where $\psi$ is a $q_1$-dimensional parameter vector of interest and $\lambda$ is a $q_2$-dimensional nuisance parameter vector, $q=q_1+q_2$. The Godambe information matrix of a information-unbiased composite likelihood is $G(\theta)=H(\theta)=J(\theta)$, and
\begin{equation*}
G(\theta)=
\left( {\begin{array}{cc}
 G_{\psi\psi}&G_{\psi\lambda}\\
 G_{\lambda\psi}&G_{\lambda\lambda}
 \end{array} } \right),
\end{equation*}
where $G_{\psi\psi}$ is the $q_1\times q_1$ submatrix of $G(\theta)$ pertaining to $\psi$, and $G_{\lambda\lambda}$ the $q_2\times q_2$ sub-matrix of $G(\theta)$ pertaining to $\lambda$. When $\lambda$ is unknown, the asymptotic variance of the MCLE of $\psi$ is given by $(G_{\psi\psi}-G_{\psi\lambda}G_{\lambda\lambda}^{-1}G_{\lambda\psi})^{-1}$; when $\lambda$ is known, the asymptotic variance of the MCLE of $\psi$ can be shown to be $G_{\psi\psi}^{-1}$. Since $G_{\psi\lambda}G_{\lambda\lambda}^{-1}G_{\lambda\psi}$ is a nonnegative matrix, we have $(G_{\psi\psi}-G_{\psi\lambda}G_{\lambda\lambda}^{-1}G_{\lambda\psi})^{-1} \geq G_{\psi\psi}^{-1}$.
However, the reverse relationship may be observed for an information-biased composite likelihood, which is illustrated through the equicorrelated multivariate normal model in the rest of this section. From previous section we know that the pairwise likelihood $CL_{pair}(\theta;\mathbf{y})$ is information-biased for this model.

\subsection{Pairwise likelihood in equicorrelated multivariate normal model}
\noindent {\bf Example 1.} Suppose $\mathbf{y}^{(1)},\ldots,\mathbf{y}^{(n)}$ are $n$ independent observations from the same $p$-dimensional multivariate normal distribution with zero mean and covariance matrix $\Sigma=\sigma^{2}\{(1-\rho)I_{p}+\rho J_{p}\}$, where $I_{p}$ is identity matrix and $J_{p}$ is a $p\times p$ matrix with all entries equal to $1$. The common correlation coefficient $\rho$ is the parameter of interest.
\medskip
The equicorrelated multivariate normal model has been well studied to compare the efficiency of pairwise likelihood and full likelihood in different settings (Arnold \& Strauss, 1991; Cox \& Reid, 2004; Mardia et al., 2009): when $\sigma^2$ is unknown, the maximum pairwise likelihood estimator of $\rho$, denoted as $\hat\rho_{pl}$, is identical to the MLE of $\rho$ and hence fully efficient; when $\sigma^2$ is known, the maximum pairwise likelihood estimator, denoted as $\tilde\rho_{pl}$, is less efficient than the maximum likelihood estimator of $\rho$ . Here we are interested in comparing the asymptotic variances of $\hat\rho_{pl}$ and $\tilde\rho_{pl}$. The asymptotic variance of $\tilde\rho_{pl}$ is ( Cox \& Reid, 2004)
 \begin{equation}
 \text{avar}(\tilde{\rho}_{pl}) = \frac{2(1-\rho)^2}{np(p-1)}\frac{c(p,\rho)}{(1+\rho^2)^2},\label{eq:known}
 \end{equation}
where $c(p,\rho)=(1-\rho)^2(3\rho^2+p^2\rho^2+1)+p\rho(-3\rho^3+8\rho^2-3\rho+2)$.
The asymptotic variance of $\hat\rho_{pl}$ can be shown to be
 \begin{equation}
 \text{avar}(\hat{\rho}_{pl}) = \frac{2(1-\rho)^2}{np(p-1)}\{1+(p-1)\rho\}^2.\label{eq:unknown}
 \end{equation}

Comparing the Equations (\ref{eq:known}) and (\ref{eq:unknown}), we find that as $\rho$ approaches its lower bound $-1/(p-1)$, $\text{avar}(\hat{\rho}_{pl})$ decreases to zero while $\text{avar}(\tilde{\rho}_{pl})$ does not. The ratio of the asymptotic variances, avar$(\tilde\rho_{pl})/\text{avar}(\hat\rho_{pl})$, as a function of $\rho$ is plotted in Figure~\ref{fig2} for $p=3$. We can see that when $\rho$ is positive, $\tilde\rho_{pl}$ is more efficient than $\hat\rho_{pl}$; when $\rho<0$, the opposite phenomenon is observed, and when $\rho$ approaches the lower bound $-0.5$, this ratio diverges to infinity. We performed the comparisons for different $p$ and observed the same phenomenon.

\begin{figure}
\begin{center}
\includegraphics[height=7.5cm,width=.72\textwidth]{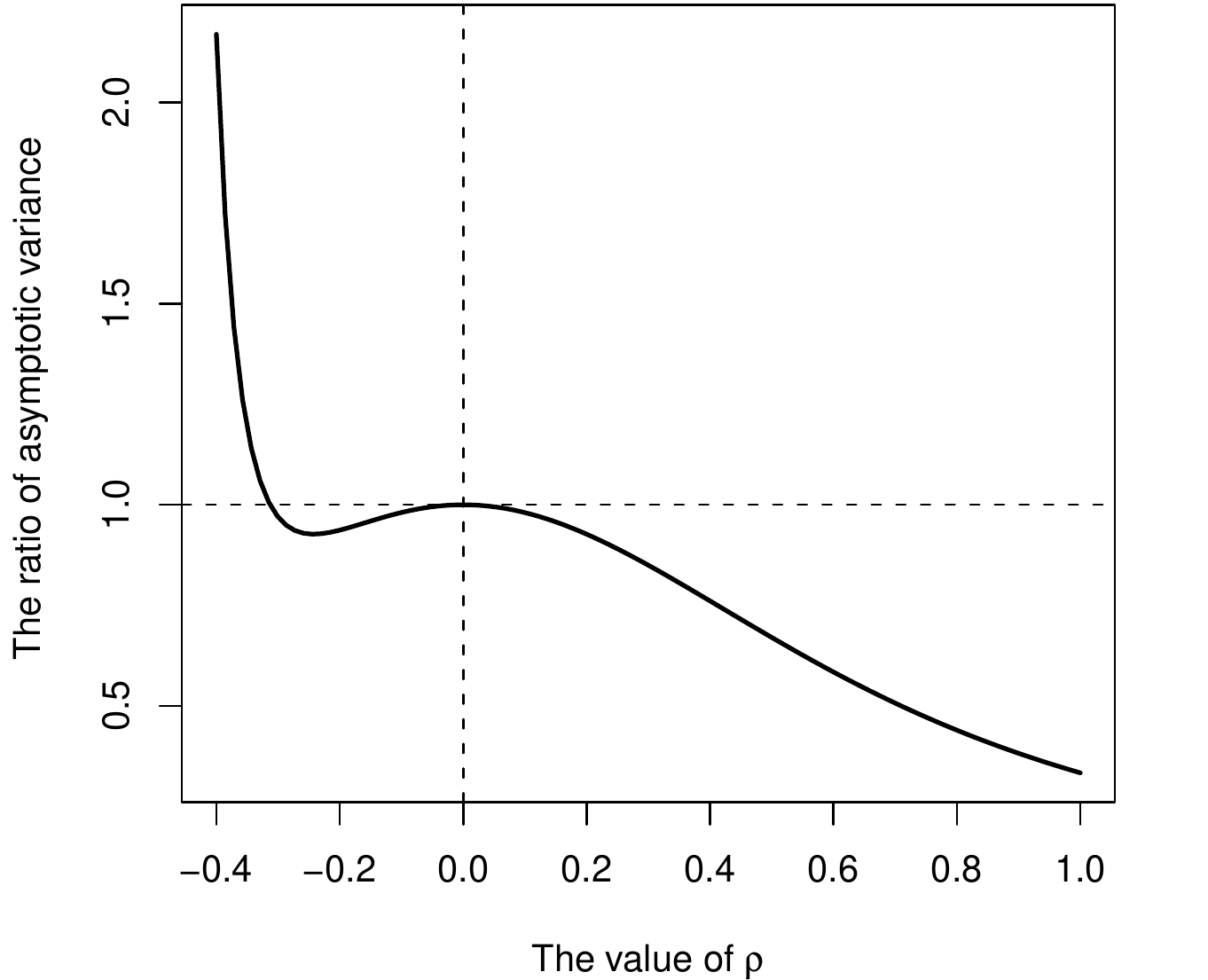}
\end{center}
\caption{The plot of the ratio $r(\rho)$= avar$(\tilde\rho_{pl})/\text{avar}(\hat\rho_{pl})$ at $p=3$. The vertical and horizontal dashed line denotes $\rho=0$ and $r(\rho)=1$ respectively.}
\label{fig2}
\end{figure}

To see that information-biasedness is not a sufficient condition for the paradox to occur, we consider another information-biased composite likelihood function, the full conditional likelihood for the same model:
\begin{equation*}
CL_{FC}(\theta;y)=\prod_{r=1}^{p}f(y_r\mid y_{(-r)};\theta\, ),
\end{equation*}
where $y_{(-r)}$ denotes the random vector excluding $y_r$. When $\sigma^2$ is unknown, the maximum full conditional likelihood estimator of $\rho$, $\hat{\rho}_{_{FC}}$ is identical to $\hat{\rho}_{pl}$ and fully efficient (Mardia et al., 2009); when $\sigma^2$ is known, the maximum full conditional likelihood estimator, $\tilde{\rho}_{_{FC}}$ is less efficient than the maximum likelihood estimator for $p\geq 3$. Using the formula in Mardia, Hughes, \& Taylor (2007), the ratio of the asymptotic variances, avar$(\tilde{\rho}_{_{FC}})/\text{avar}(\hat{\rho}_{_{FC}})$, as a function of $\rho$ is plotted in Figure~\ref{fig3.3} for $p=3$. We can see that the ratio is less than $1$ for all $\rho\in[-1/(p-1), 1]$.

\begin{figure}
\begin{center}
\includegraphics[height=8cm,width=.64\textwidth]{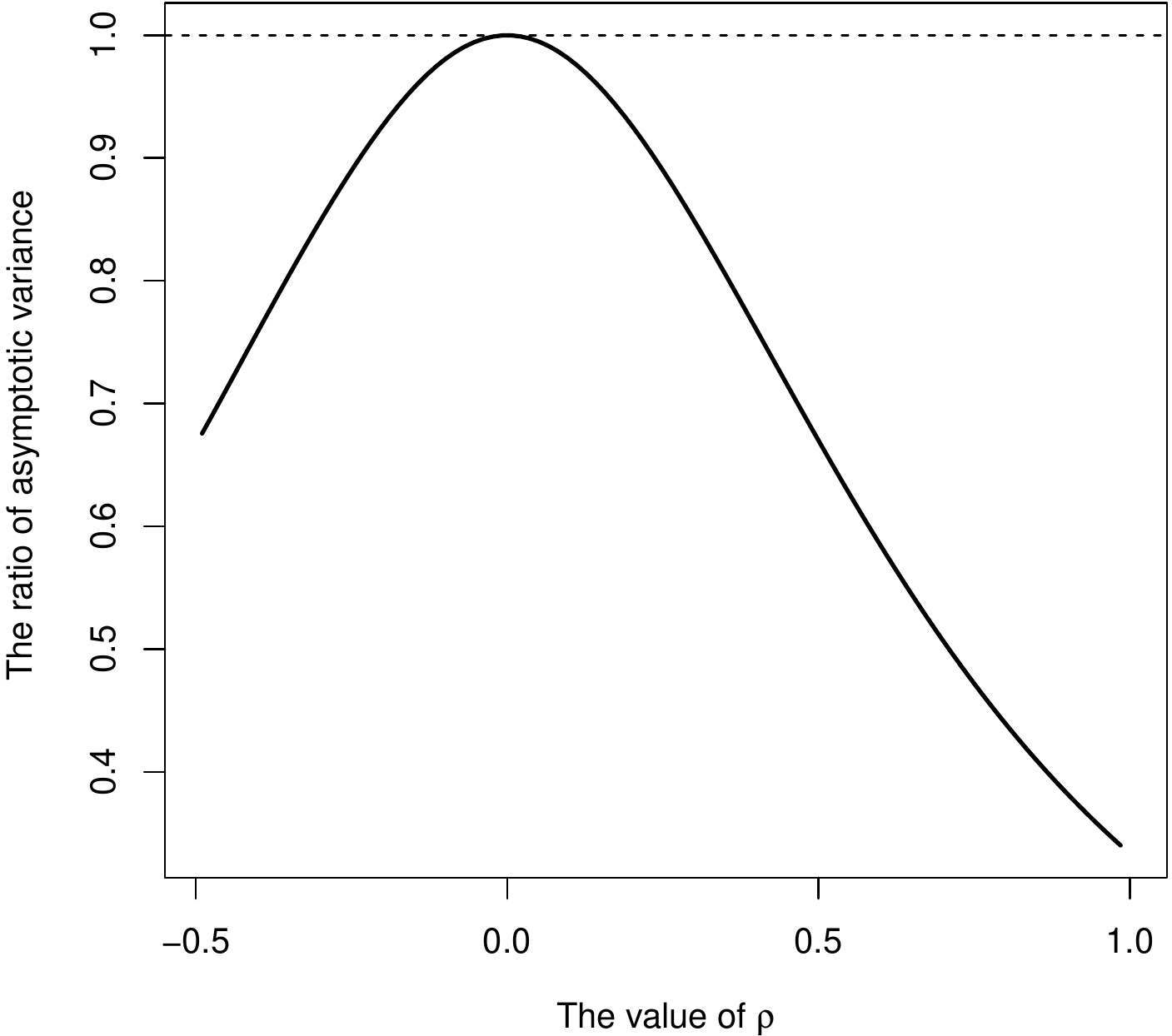}
\end{center}
\caption{The plot of the ratio $r(\rho)=\text{avar}(\tilde{\rho}_{_{FC}})/\text{avar}(\hat{\rho}_{_{FC}})$ at $p=3$. The horizontal dashed line denotes $r(\rho)=1$.}
\label{fig3.3}
\end{figure}

\subsection{A sufficient condition for the occurrence of paradox }
For the inference based on unbiased estimating equations, Henmi \& Eguchi (2004) provided a sufficient condition for the occurrence of paradox in the presence of nuisance parameters, which also applies to the information-biased composite likelihood since its score function is a special unbiased estimating equation:
\begin{theorem}{Proposition 1.}{}%
Suppose the composite likelihood function $CL(\theta\mid \mathbf{y})$ is information-biased and $\theta=(\psi, \lambda)$ with the maximum composite likelihood estimators $\hat{\theta}_{c}=(\hat{\psi}_{c}, \hat{\lambda}_{c})$. We define $\tilde{\psi}_{c}$, the MCLE of $\psi$ when $\lambda$ is known . Then, $\hat{\psi}_{c}$ has a smaller asymptotic variance than $\tilde{\psi}_{c}$ if $\hat{\lambda}_{c}$ and $\hat{\psi}_{c}$ are asymptotically independent but $\hat{\lambda}_{c}$ and $\tilde\psi_{c}$ are dependent.
\end{theorem}
In the example of equicorrelated multivariate normal model, denote by $\hat{\sigma}^2_{pl}$ the maximum pairwise likelihood estimator of $\sigma^2$, we can show that the asymptotic covariance between $\hat{\rho}_{pl}$ and $\hat{\sigma}^2_{pl}$ is $2\rho(1-\rho)\{1+(p-1)\rho\}\sigma^2/(np)$ which goes to $0$ as $\rho$ approaches $-1/(p-1)$ while the asymptotic covariance between  $\tilde\rho_{pl}$ and $\hat{\sigma}^2_{pl}$ is not equal to zero at $\rho=-1/(p-1)$. This may explain why the paradox occurs when $\rho$ is close to its lower bound $-1/(p-1)$.

\section{Composite likelihood with more component likelihoods}
In this section  we consider the inference on a single parameter only, and the simple illustrative examples allow us to calculate the Godambe information matrices or asymptotic variances analytically.

\subsection { Uncorrelated information-biased composite likelihoods}
Consider a set of information-unbiased composite likelihood functions $CL_1(\theta\mid \mathbf{y}), CL_2(\theta\mid \mathbf{y}),\ldots, CL_m(\theta\mid \mathbf{y})$ with mutually uncorrelated score functions, it is easy to show that the product $\prod_{i=1}^m CL_i(\theta\mid \mathbf{y})$ is also information-unbiased and has Godambe information matrix $G(\theta)=\sum_{i=1}^mG_i(\theta)$ where $G_i(\theta)$ is the information matrix of $CL_i(\theta\mid \mathbf{y})$. However, as shown in the example below, if any composite likelihood $CL_i(\theta\mid \mathbf{y})$ is information-biased, then information additivity may not hold for the product of uncorrelated composite likelihoods.

\noindent {\bf Example 2.} Suppose the random vector $(Y_1, Y_2, Y_3)$ follows a normal distribution with mean vector $\mu\times (1, 1, 1)^{T}$ and covariance matrix
$$\Sigma=\left( \begin{array}{ccc}
 1 & ~~~~\rho~~~~ &0  \\
 \rho &~~~~ 1~~~~ &0  \\
 0 & ~~~~0~~~~ & \sigma^2\\
 \end{array}  \right).
$$

Assume $\sigma^2$ is known, $\mu$ and $\rho$ are unknown, and $\mu$ is the only parameter of interest. Suppose we use the independence likelihood function $CL_{12}(\mu)=f(y_1;\mu)\times f(y_2;\mu)$, which is free of the nuisance parameter $\rho$, to estimate $\mu$.  To incorporate the information contained in the independent variable $Y_3$, we also consider the composite likelihood function $CL_{123}(\mu)=CL_{12}(\mu)\times f(y_3;\mu)$.

It is easy to show that the maximum composite likelihood estimators for $CL_{12}(\mu)$ and $CL_{123}(\mu)$ are $\hat{\mu}_{12}=(\bar{y}_1+\bar{y}_2)/2$  and $\hat{\mu}_{123}=\sigma^2(\bar{y}_1+\bar{y}_2)/(1+2\sigma^2)+\bar{y}_3/(1+2\sigma^2)$ with variances $(1+\rho)/2n$ and $[2(1+\rho)\sigma^4+\sigma^2]/n(1+2\sigma^2)^2$ respectively. We can compare the variances of the two maximum composite likelihood estimators directly. For example if $\sigma^2=2$, the variance of $\hat{\mu}_{123}$ is $(10+8\rho)/(25n)$ which is smaller than $(1+\rho)/(2n)$ if and only if $\rho>-5/9$.

Note that if $\rho=-1$ this result is expected as $(Y_1, Y_2)$ determines $\mu$ exactly with $\mu\equiv(Y_1+Y_2)/2$; but the dependence on $\sigma^2$ of the range of $\rho$ over which $Y_3$ degrades the inference is surprising; as $\sigma^2$ increases this range approaches $[-1, -1/2)$.

\subsection{Pairwise likelihood and independence likelihood}

Intuitively, a composite likelihood with higher dimensional component likelihoods should achieve a higher efficiency, although it usually demands more computational cost. In this subsection we focus on comparing the independence likelihood $CL_{ind}(\theta;\mathbf{y})=\prod_{r=1}^{p}f(y_r;\theta)$ and the pairwise likelihood $CL_{pair}(\theta;\mathbf{y})=\prod_{r=1}^{p-1}\prod_{s=r+1}^{p}f(y_r,y_s;\theta )$.  $CL_{pair}(\theta;\mathbf{y})$ can be written as the product of $CL_{ind}(\theta;\mathbf{y})$ and some pairwise conditional likelihood functions.  Under independence, $CL_{ind}(\theta;y)$ is identical to the full likelihood, and $CL_{pair}(\theta;y)=\{CL_{ind}(\theta;y)\}^{p-1}$, which is also fully efficient. For a multivariate normal model with continuous responses, Zhao \& Joe (2005) proved that the maximum pairwise likelihood estimator of the regression coefficient has a smaller asymptotic variance than the maximum independence likelihood estimator.

In fact, within the family of information-unbiased composite likelihood functions pairwise likelihood is  at least as efficient as independence likelihood: each bivariate density $f(y_r,y_s;\theta)$ has a larger information matrix than $f(y_r;\theta)$ and the total number of the bivariate densities in $CL_{pair}(\theta;\mathbf{y})$ is $p(p-1)/2>p$ when $p>2$.  However, the two composite likelihoods are information-biased in general especially for complex dependent data and we may observe the reverse relationship.

A bivariate binary model was used in Arnold \& Strauss (1991) to show that the pairwise conditional likelihood could be less efficient than the independence likelihood. For a bivariate model the pairwise likelihood is the full likelihood and hence fully efficient. Here we consider a four dimensional binary model which has a complex dependence structure but also allows us to compare the (asymptotic) variances of different composite likelihood estimators analytically.

\noindent {\bf Example 3.} Suppose $(Y_1, Y_2, Y_3, Y_4)$ follows a Multinomial$(1;\theta,\theta, \theta/k,1-2\theta-\theta/k)$, where $k$ is a positive constant and $0\leq\theta\leq k/(2k+1)$. The parameter $\theta$ controls both the mean and covariance structures, and we can change the value of $k$ to adjust the strength of dependence. The value of $Y_4$ is completely determined by $1-Y_1-Y_2-Y_3$. Given a random sample of size $n$ from this model, we estimate $\theta$ based on the independent triplets $(y_{1}^{(i)}, y_{2}^{(i)}, y_{3}^{(i)})$, $i=1,\ldots, n$.

The full likelihood for the model of $(Y_1, Y_2, Y_3)$ is
\begin{equation}
L(\theta)=\prod_{i=1}^{n}\theta^{y_{1}^{(i)}+y_{2}^{(i)}}(\frac{\theta}{k})^{y_{3}^{(i)}}(1-2\theta-\frac{\theta}{k})^{1-y_{1}^{(i)}-y_{2}^{(i)}-y_{3}^{(i)}}. \label{eq:eg2full}
\end{equation}
Solving the score equation we get the maximum likelihood estimator of $\theta$, $\hat\theta=(\bar{y}_1+\bar{y}_2+\bar{y}_3)/(2+1/k)$. The exact variance of $\hat\theta$ is
\begin{equation}
\text{Var}(\frac{\bar{y}_1+\bar{y}_2+\bar{y}_3}{2+1/k})=\frac{1}{n}(\frac{\theta}{2+1/k}-\theta^2).\nonumber
\end{equation}
The independence likelihood function for the model of $(Y_1, Y_2, Y_3)$ is
\begin{align}
CL_{ind}(\theta)&=\prod_{i=1}^{n}f(y_1^{(i)};\theta)f(y_2^{(i)};\theta)f(y_3^{(i)};\theta)\nonumber \\
                &=\prod_{i=1}^{n}\theta^{y_1^{(i)}}(1-\theta)^{y_1^{(i)}}\theta^{y_2^{(i)}}(1-\theta)^{y_2^{(i)}}  (\frac{\theta}{k})^{y_3^{(i)}}(1-\frac{\theta}{k})^{1-y_3^{(i)}} \label{eq:eg2ind}
\end{align}
and we can calculate its sensitivity matrix and variability matrix as
\begin{align}
H_{ind}(\theta)&=\frac{2}{\theta}+\frac{2}{1-\theta}+\frac{1}{k\theta}+\frac{1}{k(k-\theta)}, \nonumber \\ J_{ind}(\theta)&=\frac{2}{\theta(1-\theta)}+\frac{1}{\theta(k-\theta)}-\frac{2}{(1-\theta)^2}-\frac{4}{(1-\theta)(k-\theta)}.\nonumber
\end{align}
The pairwise likelihood function is
\begin{align}
CL_{pair}(\theta)&=\prod_{i=1}^{n}f(y_1^{(i)}, y_2^{(i)};\theta)f(y_1^{(i)}, y_3^{(i)};\theta)f(y_2^{(i)},y_3^{(i)};\theta)\nonumber \\
                &=\prod_{i=1}^{n}\theta^{2y_1^{(i)}+2y_2^{(i)}}(1-2\theta)^{1-y_1^{(i)}-y_2^{(i)}}  (\frac{\theta}{k})^{2y_3^{(i)}}(1-\theta-\frac{\theta}{k})^{2-y_1^{(i)}-y_2^{(i)}-2y_3^{(i)}},\label{eq:eg2pair}
\end{align}
and we can calculate its sensitivity matrix and variability matrix as
\begin{align}
H_{pair}(\theta)&=\frac{4}{\theta}+\frac{2}{k\theta}+\frac{4}{1-2\theta}+\frac{2(1+1/k)^2}{1-(1+1/k)\theta},\nonumber\\
J_{pair}(\theta)&=2A^2\theta(1-\theta)+B^2(\frac{\theta}{k})(1-\frac{\theta}{k})-2A^2\theta^2-4AB\frac{\theta^2}{k},\nonumber
\end{align}
where $A=2/\theta+2/(1-2\theta)+(1+1/k)/(1-\theta-\theta/k)$,  $B=2/\theta+2(1+1/k)/(1-\theta-\theta/k)$.

For $k=5$, the asymptotic variances of the maximum composite likelihood estimators for (\ref{eq:eg2full}), (\ref{eq:eg2ind}) and (\ref{eq:eg2pair}) multiplied by $n$ are plotted as a function of $\theta$ in Figure~\ref{fig3.1}. We can see that when $\theta<0.3$, the three estimators perform almost equally well; when $\theta>0.3$, the full likelihood becomes more efficient than the independence likelihood, and the independence likelihood estimator is more efficient than the pairwise likelihood estimator. We also carried out the comparisons for different values of $k$ and found that at $k=1$, both the independence likelihood and the pairwise likelihood are fully efficient, but when $k>1$, the independence likelihood is more efficient than the pairwise likelihood and the ratio of asymptotic variances approaches $1$ when $k\rightarrow\infty$. When $k<1$, the pairwise likelihood estimator is more efficient than the independence likelihood estimator and the ratio of asymptotic variances approaches $1$ when $k\rightarrow 0$.

\begin{figure}
\begin{center}
\includegraphics[height=7cm,width=.99\textwidth]{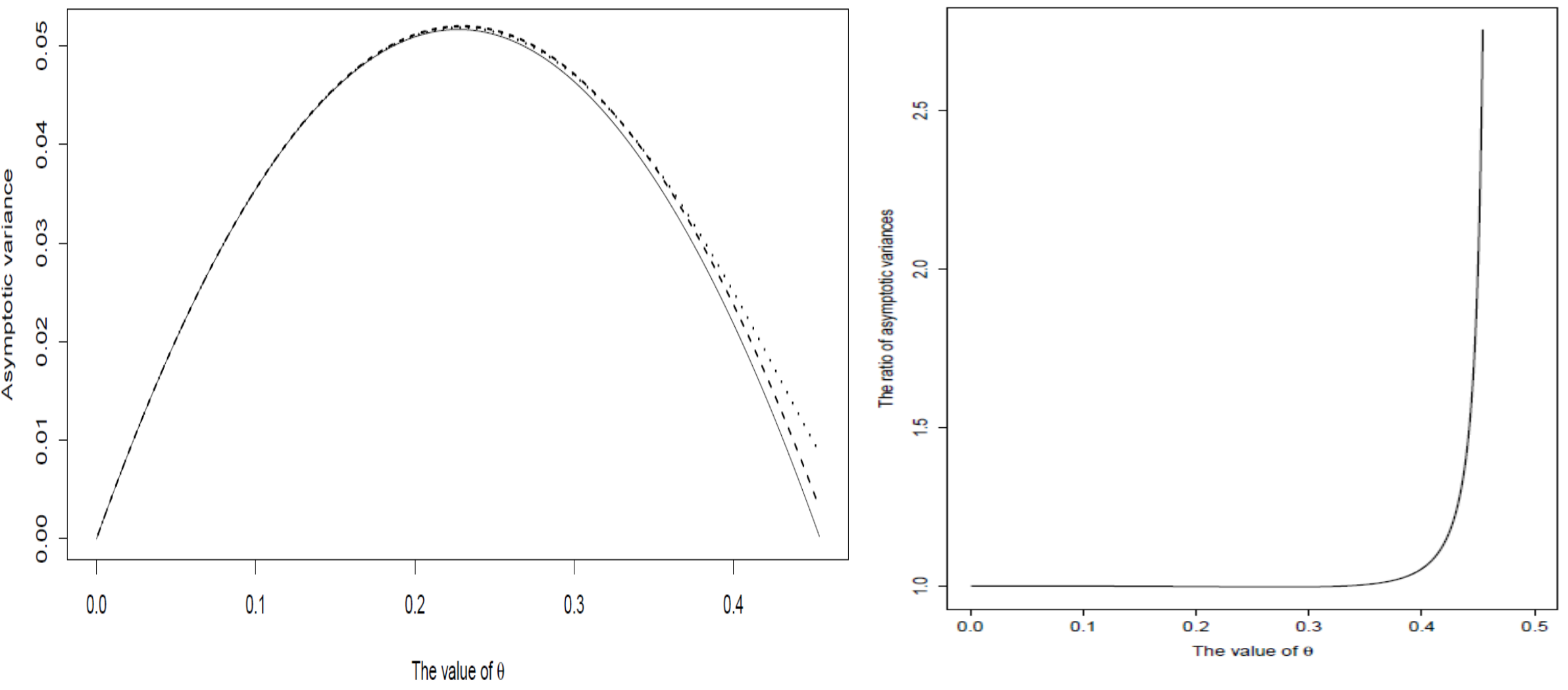}
\end{center}
\caption{Left panel: the asymptotic variances (multiplied by n) of the maximum composite likelihood estimators for the full likelihood (solid line), the independence likelihood (dashed line) and the pairwise likelihood (dotted line). Right panel: the ratio of the asymptotic variance of pairwise likelihood to the asymptotic variance of independence likelihood.}
\label{fig3.1}
\end{figure}
This example suggests that in practical applications of composite likelihood inference, where the models will usually have more complex dependence structure and incomplete data (e.g., Yi, Zeng, \& Cook, 2011), some care is required for the use of higher dimensional composite likelihood to obtain more efficient estimators. A hybrid composite likelihood combining lower dimensional marginal and conditional likelihoods with different weights is suggested to guarantee the improvement of efficiency (Cox \& Reid, 2004; Kenne Pagui, Salvan, \& Sartori, 2014a).

\section{Discussion}
As a complement to the discussion on composite likelihood inference in Reid (2012), in this paper we explored the impact of information bias on the composite likelihood based inference in different scenarios. An information-unbiased composite likelihood behaves somewhat like the ordinary likelihood but it can be very inefficient. On the other hand an information-biased likelihood brings extra difficulty to the computation and is more likely to exhibit undesirable inferential properties, although the information loss can be minimized with a set of  carefully selected weights.

One way to avoid the paradoxical phenomenon in Section 2 is to convert the composite score function $u_c(\theta;\mathbf{y})$ to an unbiased estimating function by projecting (Henmi \& Eguchi, 2004; Lindsay, Yi, \& Sun, 2011):
\begin{equation}
u_c^*(\theta;\mathbf{y})=H(\theta)J^{-1}(\theta)u_c(\theta;\mathbf{y})=\arg\min_{\nu=Au_c(\theta;\mathbf{y})}E\left\{\left\|u(\theta;\mathbf{y})-\nu(\theta;\mathbf{y})\right\|^2\right\},\label{eq:project}
\end{equation}
where $u(\theta;\mathbf{y})$ is the score function of full likelihood, $A$ ranges over all $q\times q$ matrices, $H(\theta)$ and $J(\theta)$ are the sensitivity matrix and variability matrix. It is easy to check that $u_c^*(\theta;\mathbf{y})$ is information-unbiased. Since $H(\theta)$ and $J(\theta)$ are constant matrices, this projection does not change the point estimator of $\theta$, and $u_c^*(\theta;\mathbf{y})$ has the same Godambe information as $u_c(\theta;\mathbf{y})$. In the equicorrelated multivariate normal model with $\theta=(\rho, \sigma^2)$, the score function of the pairwise likelihood is  $u_{pl}(\theta; \mathbf{y})=J(\theta)H^{-1}(\theta)u(\theta,\mathbf{y})$ (Pace, Salvan \& Sartori, 2011). From Equation (\ref{eq:project}), the projected estimating funtion of $u_{pl}(\theta; \mathbf{y})$ is equal to the score function of full likelihood, $u(\theta;\mathbf{y})$. In complex models, the required computation for the projected estimating function $u_c^*(\theta;\mathbf{y})$ can be intractable and it may be a better idea to design a nuisance-parameter-free composite likelihood function carefully for practical use. As an example, a pairwise difference likelihood that eliminates nuisance parameters in a Neyman--Scott problem is described in Hjort \& Varin (2008).

\end{spacing}
%\begin{appendix}

%\begin{proof}{Proof of Theorem 1}{}%
%We now prove the two parts of Theorem 1.
%\end{proof}
%\end{appendix}
\end{document}